\documentclass[11 pt]{amsart}
\usepackage{latexsym,amscd,amssymb, graphicx, amsthm, bm}  
\usepackage{young}
\usepackage{arydshln}
\usepackage{tikz}

\usepackage[margin=1in]{geometry}

\numberwithin{equation}{section}

\newtheorem{theorem}{Theorem}[section]
\newtheorem{proposition}[theorem]{Proposition}
\newtheorem{corollary}[theorem]{Corollary}
\newtheorem{lemma}[theorem]{Lemma}

\newtheorem{observation}[theorem]{Observation}
\newtheorem{question}[theorem]{Question}

\newtheorem{remark}[theorem]{Remark}

\newtheorem{defn}[theorem]{Definition}
\theoremstyle{definition}

\newcommand{\sign}{{\mathrm {sign}}}

\newcommand{\dinv}{{\mathrm {dinv}}}

\newcommand{\Hilb}{{\mathrm {Hilb}}}
\newcommand{\grFrob}{{\mathrm {grFrob}}}

\newcommand{\coinv}{{\mathrm {coinv}}}

\newcommand{\Inj}{{\mathrm {Inj}}}

\newcommand{\symm}{{\mathfrak{S}}}

\newcommand{\wt}{{\mathrm{wt}}}

\newcommand{\CC}{{\mathbb {C}}}
\newcommand{\QQ}{{\mathbb {Q}}}

\newcommand{\OP}{{\mathcal{OP}}}

\newcommand{\CCC}{{\mathcal{C}}}

\newcommand{\code}{{\mathtt{code}}}

\newcommand{\aaa}{{\mathbf {a}}}
\newcommand{\xx}{{\mathbf {x}}}

\newcommand{\maxcode}{{\mathtt {maxcode}}}

\newcommand{\initial}{{\mathrm {in}}}

\begin{document}

\title[Harmonic bases for generalized coinvariant algebras]
{Harmonic bases for generalized coinvariant algebras}

\author{Brendon Rhoades, Tianyi Yu, and Zehong Zhao}
\address
{Department of Mathematics \newline \indent
University of California, San Diego \newline \indent
La Jolla, CA, 92093, USA}
\email{(bprhoades, tiy059, zez045)@ucsd.edu}

\begin{abstract}
Let $k \leq n$ be nonnegative integers and let $\lambda$ be a partition of $k$.
S. Griffin recently introduced a quotient $R_{n,\lambda}$ of the polynomial ring 
$\QQ[x_1, \dots, x_n]$ in $n$ variables which simultaneously generalizes the 
Delta Conjecture coinvariant rings of Haglund-Rhoades-Shimozono
and the cohomology rings of Springer fibers studied by Tanisaki and Garsia-Procesi.
We describe the space $V_{n,\lambda}$ of harmonics attached to $R_{n,\lambda}$ 
and produce a harmonic basis of $R_{n,\lambda}$ indexed by certain ordered set partitions $\OP_{n,\lambda}$.
The combinatorics of this basis is governed by a new extension of 
the {\em Lehmer code} of a permutation to $\OP_{n, \lambda}$.
\end{abstract}

\maketitle

\section{Introduction}
\label{Introduction}

In his Ph.D. thesis \cite{Griffin}, Sean Griffin introduced the following remarkable family of quotients of the polynomial
ring $\QQ[\xx_n] := \QQ[x_1, \dots, x_n]$ in $n$ variables.
Given a subset $S \subseteq [n] := \{1, 2, \dots, n \}$ and $d \geq 0$, let $e_d(S)$ be the 
degree $d$ elementary
symmetric polynomial in the variable set $\{ x_i \,:\, i \in S \}$.  For example, we have
\begin{equation*}
e_2(1457) = x_1 x_4 + x_1 x_5 + x_1 x_7 + x_4 x_5 + x_4 x_7 + x_5 x_7.
\end{equation*}
By convention, we set $e_d(S) = 0$ whenever $|S| > d$.

For $k \geq 0$,
in this paper we use the term {\em partition of $k$} to mean a weakly decreasing sequence
$\lambda = (\lambda_1 \geq \cdots \geq \lambda_s)$ of {\bf nonnegative} integers
with $\lambda_1 + \cdots + \lambda_s = k$.  We write $|\lambda| = k$ or $\lambda \vdash k$
to mean that $\lambda$ is a partition of $k$ and call $s$ the number of {\em parts} of 
$\lambda = (\lambda_1, \dots, \lambda_s)$.
We also write $\ell(\lambda)$ for the number of {\em nonzero} parts of $\lambda$.
For example, if $\lambda = (4,2,2,0,0)$ we have $\lambda \vdash 8,$ the partition $\lambda$
has $5$ parts, and $\ell(\lambda) = 3$.

  \begin{defn} (Griffin \cite{Griffin})
  Let $k \leq n$ be nonnegative integers and let $\lambda = (\lambda_1 \geq \cdots \geq \lambda_{s} \geq 0)$
be a partition of $k$ with $s$ parts.
Write $\lambda' = (\lambda'_1 \geq \cdots \geq \lambda'_n \geq 0)$ for the conjugate partition of 
$\lambda$, padded with trailing zeros to be of length $n$.

  Let $I_{n,\lambda} \subseteq \QQ[\xx_n]$ be the ideal
  \begin{equation}
  I_{n,\lambda} = \langle x_1^{s}, \dots, x_n^{s} \rangle + 
  \langle e_d(S) \,:\, S \subseteq [n] \text{ and } d > |S| - \lambda'_n - \lambda'_{n-1} - \cdots - \lambda'_{n-|S|+1} \rangle
  \end{equation}
  and let $R_{n,\lambda} := \QQ[\xx_n]/I_{n,\lambda}$ be the associated quotient ring
  \footnote{Our notation differs from that of Griffin \cite{Griffin}. He considers partitions to only 
  have positive parts and includes $s \geq \ell(\lambda)$ as a parameter, so that his objects 
  are denoted $I_{n,\lambda,s}$ and $R_{n,\lambda,s}$. Griffen uses the symbol $I_{n,\lambda}$ to denote
  the ideal $\langle e_d(S) \,:\, S \subseteq [n] \text{ and } d > |S| - \lambda'_n - \lambda'_{n-1} - \cdots - \lambda'_{n-|S|+1} \rangle$ and $R_{n,\lambda}$ to denote the corresponding quotient - this is the `limit' as $s \rightarrow \infty$ of
  his $I_{n,\lambda,s}$ and $R_{n,\lambda,s}$.}.
  \end{defn}
  
  As an example, suppose $n = 9$ and $\lambda = (3,2,2,0)$ so that $k = 7$ and $s = 4$.
  The conjugate partition $\lambda' = (\lambda'_1, \dots, \lambda'_9)$ is given by
  $(3,3,1,0,0,0,0,0,0)$.
  The ideal $I_{9,\lambda} \subseteq \QQ[\xx_9]$ is generated by $x_1^4, \dots, x_9^4$ together with
  the polynomials
  \begin{align*}
  &e_i(S) \quad \quad S \subseteq [9] \quad \quad |S| = 9 \quad \quad i = 9,8,7,6,5,5,3,  \\
  &e_j(T) \quad \quad T \subseteq [9] \quad \quad |T| = 8 \quad \quad j = 8,7,6,5, \\
  &e_{d}(U) \quad \quad U \subseteq [9] \quad \quad |U| = 7 \quad \quad d = 7,6.\\
  \end{align*}
  Griffin's rings $R_{n,\lambda}$  generalize several important classes of quotient rings 
  in algebraic combinatorics.
  \begin{itemize}
  \item When $k = s = n$ and $\lambda = (1^n)$, the ideal $I_{n,\lambda}$ is generated by the 
  $n$ elementary symmetric polynomials $e_1(\xx_n), e_2(\xx_n), \dots, e_n(\xx_n)$ in the full variable set 
  $\{x_1, \dots, x_n \}$
  and $R_{n,\lambda}$ is the classical {\em coinvariant ring} 
  \begin{equation}
  R_n := \QQ[\xx_n]/ \langle e_1(\xx_n), e_2(\xx_n), \dots, e_n(\xx_n) \rangle
  \end{equation}
  attached to the symmetric group $\symm_n$.
  The ring $R_{n,\lambda}$ presents the cohomology of the complete flag variety of type A$_{n-1}$.
  \item When $k = n$ and $\lambda \vdash n$ is arbitrary,
  the ring $R_{n,\lambda}$ is the {\em Tanisaki quotient} studied
  by Tanisaki \cite{Tanisaki} and Garsia-Procesi \cite{GP} which presents the cohomology of the
  {\em Springer fiber} $\mathcal{B}_{\lambda}$ attached to the partition $\lambda$. 
  \item When $\lambda = (1^k, 0^{s-k})$ has all parts $\leq 1$, the rings $R_{n,\lambda}$ were 
  introduced by Haglund, Rhoades, and Shimozono \cite{HRS} to give a representation-theoretic model for
  the
  Haglund-Remmel-Wilson {\em Delta Conjecture} \cite{HRW}. 
 Pawlowski-Rhoades proved that these rings present the cohomology
the moduli space of 
$n$-tuples of lines  $(\ell_1, \dots, \ell_n)$ in $\CC^{s}$
such that the composite linear map
\begin{equation}
\ell_1 \oplus \cdots \oplus \ell_n \rightarrow \CC^s \twoheadrightarrow \CC^k
\end{equation}
given by addition $(v_1, \dots, v_n) \mapsto v_1 + \cdots + v_n$ followed by projection onto the first $k$
coordinates is a surjection
 \cite{PR}.
  \end{itemize}

  The symmetric group $\symm_n$ acts on $\QQ[\xx_n]$ by subscript permutation.
  The ideals $I_{n,\lambda}$ are graded and $\symm_n$-stable, so $R_{n,\lambda}$ is a graded
  $\symm_n$-module.
  Generalizing results from \cite{GP, HRS}, Griffin calculated \cite{Griffin} the graded $\symm_n$-isomorphism
  type of $R_{n,\lambda}$.
  It is an open problem to find a variety $X_{n,\lambda}$ whose cohomology is presented by $R_{n,\lambda}$,
  but Griffin proved that $R_{n,\lambda}$ is the coordinate ring of a scheme-theoretic intersection
  arising from {\em rank varieties} \cite{Griffin}.

In this paper we study the rings $R_{n,\lambda}$ as graded $\QQ$-vector spaces.  In the special case
$k = s = n$ and $\lambda = (1^n)$, the classical coinvariant ring $R_n$
has a number of interesting bases which are important for different reasons.
Perhaps the simplest of these was discovered by E. Artin \cite{Artin}, who used Galois Theory to prove that 
the family of `sub-staircase monomials'
\begin{equation}
\{ x_1^{c_1} x_2^{c_2} \cdots x_n^{c_n} \,:\, 0 \leq c_i < n - i \}
\end{equation}
descends to a basis for $R_n$.
Extending earlier results of \cite{GP, HRS}, Griffin discovered the appropriate
generalization of `sub-staircase' to obtain a monomial basis
of $R_{n,\lambda}$; his result is quoted in Theorem~\ref{griffin-basis-theorem} below. 

Our main goal in this paper is to describe the {\em harmonic space} of the quotient ring $R_{n,\lambda}$
and so derive a {\em harmonic basis} of this quotient ring.
In order to motivate harmonic spaces and bases, we recall some  technical issues
that  arise in the study of  quotient rings.

Let  $I \subseteq \QQ[\xx_n]$
be any homogeneous ideal with  quotient ring $R = \QQ[\xx_n]/I$.
In algebraic combinatorics, 
one is often interested in calculating algebraic invariants of $R$ such as its  dimension or Hilbert
series.
A frequent impediment to computing these invariants is that, given $f \in \QQ[\xx_n]$, it can be difficult
to decide whether $f + I = 0$ in $R$.
Harmonic spaces can be used to replace quotients with subspaces, circumventing this problem.

For $f = f(x_1, \dots, x_n) \in \QQ[\xx_n]$, let $\partial f := f(\partial/\partial x_1, \dots, \partial/\partial x_n)$
be the differential operator on $\QQ[\xx_n]$ obtained by replacing each $x_i$ appearing in $f$ with the 
partial derivative $\partial/\partial x_i$.
The ring  $\QQ[\xx_n]$ acts on itself by
\begin{equation}
f \odot g := (\partial f)(g) \quad \text{for all $f, g \in \QQ[\xx_n]$.}
\end{equation}
That is, the polynomial $f \odot g$ is obtained by first turning $f$ into a differential operator $\partial f$,
and then applying $\partial f$ to $g$.

For $f, g \in \QQ[\xx_n]$, we define a number $\langle f, g \rangle \in \QQ$ by
\begin{equation}
\langle f, g \rangle := \text{constant term of $f \odot g$}.
\end{equation}
Given two monomials $x_1^{a_1} \cdots x_n^{a_n}$ and $x_1^{b_1} \cdots x_n^{b_n}$, it follows directly that
\begin{equation}
\langle x_1^{a_1} \cdots x_n^{a_n}, x_1^{b_1} \cdots x_n^{b_n} \rangle = \begin{cases}
a_1! \cdots a_n! & \text{if $a_i = b_i$ for all $i$,} \\
0 & \text{otherwise,}
\end{cases}
\end{equation}
so that $\langle -, - \rangle$ is an inner product on $\QQ[\xx_n]$ for which the  degree 
grading $\QQ[\xx_n] = \bigoplus_{d \geq 0} \QQ[\xx_n]_d$ is an orthogonal decomposition.

For a homogeneous ideal $I \subseteq \QQ[\xx_n]$, the {\em harmonic space} $V$ of $I$
is the graded subspace of $\QQ[\xx_n]$ given by
\begin{equation}
V = I^{\perp} = \{ g \in \QQ[\xx_n] \,:\, \langle f, g \rangle = 0 \text{ for all $f \in I$} \}.
\end{equation}
Writing $R = \QQ[\xx_n]/I$, standard results of linear algebra imply that
$\QQ[\xx_n] = V \oplus I$ so that any vector space basis for $V$ projects onto a basis of $R$.
Any basis of $V$ (and its image basis in $R$) is called a {\em harmonic basis}.
If the ideal $I$ is $\symm_n$-invariant, the $\symm_n$-invariance of the inner product
$\langle -, - \rangle$ furnishes an isomorphism of graded $\symm_n$-modules $R \cong V$.
The harmonic space $V$ therefore permits the study of the quotient ring $R$ without the computational
issues inherent in cosets.

\begin{defn}
Let $V_{n,\lambda} \subseteq \QQ[\xx_n]$ be the harmonic space of $I_{n,\lambda}$.
\end{defn}

We have an isomorphism of graded $\symm_n$-modules $R_{n,\lambda} \cong V_{n,\lambda}$
and any basis for $V_{n,\lambda}$ descends to a basis for $R_{n,\lambda}$.
In the classical case $k = s = n$ and $\lambda = (1, \dots, 1)$ so that 
$R_{n,\lambda} = R_n$, the harmonic space $V_{n,\lambda}$ has the following description.

Recall that the {\em Vandermonde determinant} $\delta_n \in \QQ[\xx_n]$ is the polynomial
\begin{equation}
\delta_n := \prod_{1 \leq i < j \leq n} (x_i - x_j).
\end{equation}
The harmonic space $V_n \subseteq \QQ[\xx_n]$ corresponding
to $R_n$ is generated by $\delta_n$ as a $\QQ[\xx_n]$-module.
More explicitly, the space $V_n$ is the smallest subspace of $\QQ[\xx_n]$ containing $\delta_n$
which is closed under the partial derivatives $\partial/\partial x_1, \dots, \partial/\partial x_n$.
A harmonic basis of $R_n$ is given by applying sub-staircase monomials (as differential operators)
to $\delta_n$:
\begin{equation}
\{ (x_1^{c_1} \cdots x_n^{c_n}) \odot \delta_n \,:\, 0 \leq c_i \leq n - i \}.
\end{equation}
In the Springer fiber case $k = n$ with $\lambda$ arbitrary,
the harmonic space $V_{n,\lambda}$ was described by N. Bergeron and Garsia \cite{BG} 
using `partial Vandermonde' polynomials.

In order to describe our results, we need one more definition.
Given $k \leq n$ and a partition 
$\lambda = (\lambda_1 \geq \cdots \geq \lambda_s)$ of $k$, let $\OP_{n,\lambda}$ 
be the collection of length $s$ sequences $\sigma = (B_1 \mid \cdots \mid B_s)$ of subsets of $[n]$
such that 
\begin{itemize}
\item we have a disjoint union decomposition $[n] = B_1 \sqcup \cdots \sqcup B_s$, and
\item the set $B_i$ has at least $\lambda_i$ elements.
\end{itemize}
We refer to elements $\sigma \in \OP_{n,\lambda}$ as {\em ordered set partitions},
even though some of the sets $B_i$ may be empty when the partition $\lambda$ has trailing zeros.
When $k = s = n$ and $\lambda = (1^n)$, we have an identification $\OP_{n,\lambda} = \symm_n$
of ordered set partitions and permutations.

\begin{itemize}
\item  We generalize work of Rhoades-Wilson \cite{RW2} to define a coinversion statistic $\coinv(\sigma)$
on $\OP_{n,\lambda}$ and an extension $\code(\sigma)$ of the Lehmer code of a permutation 
to $\OP_{n,\lambda}$ (Section~\ref{Coinversion}).
We show that the map $\sigma \mapsto \code(\sigma)$ bijects $\OP_{n,\lambda}$ with a family 
of sequences $\CCC_{n,\lambda}$ determined by 
$n$ and $\lambda$ (Theorem~\ref{code-is-bijection}).
\item We show that the Hilbert series of $R_{n,\lambda}$ 
is the generating function of the statistic $\coinv$ on $\OP_{n,\lambda}$
(Corollary~\ref{hilbert-series}).
\item We describe a generating set for the 
harmonic space $V_{n,\lambda}$ as a $\QQ[\xx_n]$-module
(Theorem~\ref{generating-harmonic-set}) and give an explicit harmonic basis
$\{ \delta_{\sigma} \,:\, \sigma \in \OP_{n,\lambda} \}$ of $R_{n,\lambda}$ indexed by ordered set partitions
in
$\OP_{n,\lambda}$
(Theorem~\ref{harmonic-basis}).
\item We show that the lexicographical leading monomials of the harmonic polynomials in $V_{n,\lambda}$
are precisely those with exponent sequences lying in $\CCC_{n,\lambda}$
(Theorem~\ref{leading-term-theorem}).
\end{itemize}

The rest of the paper is organized as follows.
In {\bf Section~\ref{Background}} we give background on partitions, tableaux, and ordered set partitions.
In {\bf Section~\ref{Coinversion}} we describe our new coinversion statistic on
$\OP_{n,\lambda}$ as well as its associated coinversion code.
We use an insertion algorithm to describe the possible coinversion codes of elements of $\OP_{n,\lambda}$.
In {\bf Section~\ref{Harmonic}} we study the harmonic space $V_{n,\lambda}$ and, in particular,
give a harmonic basis for $V_{n,\lambda}$ (or $R_{n,\lambda}$) indexed by 
$\OP_{n,\lambda}$. 
We also prove a conjecture of A. T. Wilson (personal communication) by showing that a certain family
$R_{n,k,s}$ of graded rings coincides with a special case of Griffin's rings $R_{n,\lambda}$.
We close in {\bf Section~\ref{Conclusion}} by proposing a connection between this work and superspace.

\section{Background}
\label{Background}

\subsection{Partitions and tableaux}
Given $k \geq 0$, a {\em partition} of $n$ is a weakly decreasing sequence
$\lambda = (\lambda_1 \geq \cdots \geq \lambda_s)$ of nonnegative integers satisfying
$\lambda_1 + \cdots + \lambda_s = k$.  
In particular, we  allow trailing zeros in our partitions.  Let $\ell(\lambda)$ denote the number of nonzero
parts of a partition $\lambda$.
We use the notation $\lambda \vdash k$ to indicate that $\lambda$ is a partition of $k$.

The {\em Young diagram} of a partition $\lambda$ consists of $\lambda_i$ left-justified boxes in row $i$.
For example, the Young diagram of $(4,2,1) \vdash 7$ is shown below.
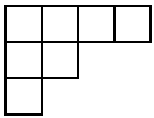
\begin{equation*}
\begin{Young}
 & & & \cr
 & \cr
 \cr
 \end{Young}
\end{equation*}
Observe that trailing zeros have no effect on Young diagrams, so this would also be the Young diagram
of the partition $(4,2,1,0,0)$.
The {\em conjugate} $\lambda'$ of a partition $\lambda$ is obtained by reflecting its Young diagram
across the main diagonal; we have $\lambda' = (3,2,1,1)$ in this case.

Let $\lambda$ be a partition.
A {\em tableau} $T$ of shape $\lambda$ is a filling $T: \lambda \rightarrow \{1, 2, \dots \}$
of the boxes of $\lambda$ with positive integers.
A tableau $T$ is  {\em column strict} if its entries strictly increase going down columns and
{\em injective} if its entries are distinct.
We write $\Inj(\lambda, \leq n)$ for the family of injective and column strict tableaux of shape $\lambda$ whose
 entries are bounded above by $n$.
 An example tableau in $\Inj((4,2,1), \leq 9)$ is shown below; observe that the number 8 does not appear 
 in this tableau.

\begin{equation*}
\begin{young}
2 & 1 & 3 & 9 \cr
5 & 4 \cr
6
\end{young}
\end{equation*}

\subsection{Ordered set partitions}
A {\em (weak) ordered set partition} of $[n]$ is a sequence $\sigma = (B_1 \mid \cdots \mid B_s)$ of 
(possibly empty)
subsets of $[n]$ such that we have a disjoint union decomposition
$[n] = B_1 \sqcup \cdots \sqcup B_s$.  We say that $\sigma$ has {\em $s$ blocks}.
As an example,
\begin{equation*}
\sigma = ( 1, \, 3, \, 5, \, 9 \, \mid \, 6, \, 7, \, 8, \, 10, \, 14 \, \mid  \, 2, \, 12, \, 15 \, \mid  4, \, 13 \, \mid \, \varnothing \, \mid
\, 11, \, 16 )
\end{equation*}
is an ordered set partition of $[16]$ with 6 blocks.

Let $\lambda = (\lambda_1 \geq \cdots \geq \lambda_s)$ be a partition. 
As described in the introduction, 
we write $\OP_{n,\lambda}$ for the family of all ordered set partitions $\sigma = (B_1 \mid \cdots \mid B_s)$ 
of $[n]$ with $s$ blocks such that $B_i$ has at least $\lambda_i$ elements, for all $1 \leq i \leq s$.
If $\lambda = (3,3,2,2,0,0)$, the ordered set partition $\sigma$ above lies in $\OP_{16, \lambda}$.

It will be convenient to visualize elements of $\OP_{n,\lambda}$ in terms of the following
{\em container diagrams}.  Given $\lambda = (\lambda_1 \geq \cdots \geq \lambda_s)$, we first draw
(from left to right) $s$ columns of top-justified boxes of height $\lambda_i$.
(These boxes are called the {\em container}.)
For $\sigma = (B_1 \mid \cdots \mid B_s) \in \OP_{n,\lambda}$, we fill the $i^{th}$ column with the entries
of $B_i$, increasing from bottom to top.
Our example ordered set partition $\sigma \in \OP_{16,(3,3,2,2,0,0)}$ has the following container diagram,
with column numberings corresponding to block indices.
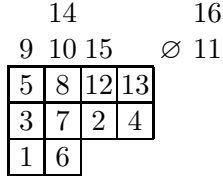
\begin{equation}
\label{example-sigma}
\begin{Young}
 ,& ,14 & ,&,  &,  &,16  \cr
 ,9 & ,10 & ,15 &, &, \varnothing  &, 11 \cr
 5 & 8 & 12 & 13 &, &,   \cr
 3 & 7 & 2 & 4 &, &, \cr
1   & 6 &, &, &, &,  \cr 
, \cr
,1 & ,2 & ,3 & ,4 & ,5 & ,6 
\end{Young}
\end{equation}
Empty blocks in ordered set partitions give rise to empty columns in container diagrams.
The container diagram above has a single empty column, decorated with the placeholder $\varnothing$.
The condition $\sigma \in \OP_{n,\lambda}$ corresponds to the container of boxes being completely filled with
numbers.
The numbers appearing outsider of the container (9, 10, 11, 14, 15, and 16 in our example) are called
{\em floating}.

\section{Coinversion codes for $\OP_{n,\lambda}$}
\label{Coinversion}

\subsection{Coinversions in ordered set partitions}
One variant of the {\em Lehmer code} of a permutation $\pi = \pi_1 \dots \pi_n \in \symm_n$
is given by the sequence $(c_1, \dots, c_n)$ where
\begin{equation}
c_i = | \{ i < j \,:\, \pi_i < \pi_j \} |.
\end{equation}
The sum of this sequence $c_1 + \cdots + c_n$ counts the total number of {\em coinversions} 
(i.e. non-inversions)
of $\pi$.
We extend this definition from permutations to ordered set partitions as follows.

Let $\lambda = (\lambda_1 \geq \cdots \geq \lambda_s) \vdash k$ be a partition, let $n \geq k$, and 
let $\sigma \in \OP_{n,\lambda}$. We think of $\sigma$ in terms of its container diagram.
For $1 \leq i < j \leq n$, we say that $(i,j)$ is a {\em coinversion} of $\sigma$ when one of the following 
three conditions hold:
\begin{itemize}
\item $i$ is floating, $j$ is to the right of $i$ in $\sigma$, $j$ is at the top of its container, and $i < j$,
\item $i$ is not floating, $j$ is to the right of $i$ in $\sigma$, $i$ and $j$ are in the same row of $\sigma$,
and $i < j$, or
\item $i$ is not floating, $j$ is to the left of $i$ in $\sigma$, $j$ is one row below $i$ in $\sigma$, and $i < j$.
\end{itemize}
The last two conditions may be depicted schematically as 
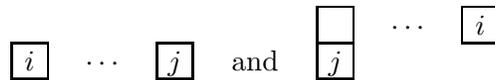
\begin{equation*}
\begin{Young}
i &, &, \cdots &, &  j 
\end{Young}  \hspace{0.2in}   \text{and}  \hspace{0.2in}
\begin{Young}
 &  ,&, \cdots &, &i \cr
 j
\end{Young}
\end{equation*}

\begin{remark}
The conditions defining coinversions 
for non-floating indices are the same as those used to define the statistic $\dinv$ which arises
in the Haglund-Haiman-Loehr monomial expansion of the modified Macdonald polynomials \cite{HHL}.
\end{remark}

For $1 \leq i \leq n$ we define a number $c_i \geq 0$ by
\begin{equation}
c_i := \begin{cases}
| \{ i < j \,:\, \text{$(i,j)$ is a coinversion of $\sigma$} \} |
& \text{if $i$ is not floating} \\
| \{ i < j \,:\, \text{$(i,j)$ is a coinversion of $\sigma$} \} | + (p-1) 
& \text{if $i$ is floating in the $p^{th}$ block of $\sigma$.}
\end{cases}
\end{equation}
The {\em coinversion code} of $\sigma$ is given by
\begin{equation}
\code(\sigma) := (c_1, \dots, c_n)
\end{equation}
and the {\em coinversion number} of $\sigma$ is given by
\begin{equation}
\coinv(\sigma) := c_1 + \cdots + c_n.
\end{equation}
Rhoades and Wilson \cite{RW2} defined $\code(\sigma)$ in the special case
where $\lambda = (1^k)$.

As an example of these concepts, consider the ordered set partition
$\sigma \in \OP_{16,(3,3,2,2,0,0)}$ appearing in \eqref{example-sigma}.
Let $(c_1, \dots, c_{16})$ be the sequence $\code(\sigma)$.
The entry $2$ forms coinversions with $4$ and $6$, so that
$c_2 = 2$.
The entry 10 is floating in column 2, and forms coinversions with 12 and 13
so that $c_{10} = 2 + (2-1) = 3$.  We have
\begin{equation*}
\code(\sigma) = (c_1, \dots, c_{16}) = (1, 2, 2, 1, 3, 0, 0, 2, 2, 3, 5, 1, 0, 1, 2, 5).
\end{equation*}
Adding this sequence yields $\coinv(\sigma) = 30$.

\subsection{The family of sequences $\CCC_{n,\lambda}$}
The map $\sigma \mapsto \code(\sigma)$ assigning $\sigma \in \OP_{n,\lambda}$
to its coinversion code will turn out to be an injection.
In order to describe the image of this map, we recall that a {\em shuffle} of two sequences
$(a_1, \dots, a_p)$ and $(b_1, \dots, b_q)$ is an interleaving $(c_1, \dots, c_{p+q})$ of these 
sequences which preserves the relative order of the $a$'s and the $b$'s. 
A shuffle of any finite number of sequences may be defined analogously (or inductively).

Let $k \leq n$ be positive integers,
let $\lambda = (\lambda_1 \geq \cdots \geq \lambda_s) \vdash k$ be a partition with $s$ nonnegative parts, 
and write the conjugate of $\lambda$ as 
$(\lambda_1' \geq \cdots \geq \lambda_k')$.
We define $\CCC_{n,\lambda}$ to be the family of length $n$ sequences $(c_1, \dots, c_n)$ of nonnegative 
integers which are componentwise $\leq$ some shuffle of the $k+1$ (possibly empty) sequences
\begin{equation*}
(\lambda'_1 - 1, \lambda'_1 - 2, \dots, 1, 0), \dots, (\lambda'_k - 1, \lambda'_k - 2, \dots, 1, 0), \text{ and }
(s-1, s-1, \dots, s-1),
\end{equation*}
where the final sequence has $n-k$ copies of $s-1$.

Continuing our running example of $n = 16$ and $\lambda = (3,3,2,2,0,0) \vdash 8$,
the nonzero parts of $\lambda'$ are $(4,4,2)$ so that $\CCC_{n,\lambda}$ consists of all length
16 sequences $(c_1, \dots, c_{16})$ of nonnegative integers which are componentwise $\leq$ 
some shuffle of the  sequences
\begin{equation*}
(3,2,1,0), \, (3,2,1,0), \, (1,0), \text{ and } (5,5,5,5,5,5).
\end{equation*}

The sequence family $\CCC_{n,\lambda}$ was introduced by Haglund-Rhoades-Shimozono \cite{HRS} in the case
$\lambda_1 \leq 1$ and by 
Griffin \cite{Griffin} for general $\lambda$.
Griffin proved that the monomials in $\QQ[\xx_n]$ whose exponent sequences lie in $\CCC_{n,\lambda}$ descend
to a basis of the ring $R_{n,\lambda}$.

\begin{theorem} (Griffin \cite{Griffin})
\label{griffin-basis-theorem}
Let $k \leq n$ be positive integers and let $\lambda = (\lambda_1 \geq \cdots \geq \lambda_s) \vdash k$.
The set of monomials
\begin{equation}
\{ x_1^{c_1} \cdots x_n^{c_n} \,:\, (c_1, \dots, c_n) \in \CCC_{n,\lambda} \}
\end{equation}
descends to a vector space basis of $R_{n,\lambda}$.
\end{theorem}

We will prove that $\code$ is a bijection from $\OP_{n,\lambda}$ to $\CCC_{n,\lambda}$.
As a first step, we show that $\code(\sigma) \in \CCC_{n,\lambda}$ for any $\sigma \in \OP_{n,\lambda}$.

\begin{lemma}
\label{code-map-is-defined}
Let $k \leq n$ be positive integers and let $\lambda = (\lambda_1 \geq \cdots \geq \lambda_s) \vdash k$
be a partition of $k$. For any $\sigma \in \OP_{n,\lambda}$ we have
$\code(\sigma) \in \CCC_{n,\lambda}$.
\end{lemma}

\begin{proof}
The $i^{th}$ row from the top of the container of $\sigma$ contains $\lambda'_i$ boxes.
It follows from the definition of coinversions that the $j^{th}$ smallest entry in this row forms at most
$\lambda'_i - j$ coinversions with other entries of $\sigma$.
Furthermore, if $t$ is any floating entry of $\sigma$, then $c_t \leq s-1$ by construction.
The entries in the $\lambda_1$ rows of $\sigma$, together with the $n-k$ floating entries, define 
a shuffle $(c'_1, \dots, c'_n)$ of the sequences
\begin{equation*}
(\lambda'_1 - 1, \lambda'_1 - 2, \dots, 1, 0), \dots, (\lambda'_k - 1, \lambda'_k - 2, \dots, 1, 0), \text{ and }
(\overbrace{s-1, s-1, \dots, s-1}^{n-k}),
\end{equation*}
such that we have the componentwise inequality
$\code(\sigma) \leq (c'_1, \dots, c'_n)$.

To see how this works, suppose $\sigma$ is as in \eqref{example-sigma}:
\begin{equation*}
\begin{Young}
 ,& ,&  ,& ,14 & ,&,  &,  &,16  \cr
 ,& ,& ,9 & ,10 & ,15 &, &, \varnothing  &, 11 \cr
,\sigma = & ,& 5 & 8 & 12 & 13 &, &,   \cr
,&, &  3 & 7 & 2 & 4 &, &, \cr
,& , & 1   & 6 &, &, &, &,  \cr 
\end{Young}
\end{equation*}
We use the container diagram of $\sigma$ to form a shuffle $(c'_1, \dots, c'_{16})$
of  the sequences 
\begin{equation*}
(3^{\bullet},2^{\bullet},1^{\bullet},0^{\bullet}), (3^{\circ},2^{\circ},1^{\circ},0^{\circ}), 
(1^{\square},0^{\square}), \text{ and } (5,5,5,5,5,5).
\end{equation*}
Here we label our sequences with decorations ($\bullet, \circ, \square,$ and unadorned) so that we 
can distinguish them when we perform our shuffle.
The shuffle $(c'_1, \dots, c'_{16})$ corresponding to $\sigma$ is:
\begin{equation*} (c'_1, \dots c'_{16}) := 
( 1^{\square} , 3^{\circ} , 2^{\circ} , 1^{\circ} , 3^{\bullet} , 0^{\square} , 0^{\circ} , 2^{\bullet} , 5 , 5 , 5 , 1^{\bullet} , 0^{\bullet} , 5 , 5 , 5 ).
\end{equation*}
The positions of the $\bullet$ entries are given by top row of the container
of $\sigma$ ($5 ,8, 12,$ and $13$).  The positions of the $\circ$ entries (2, 4, 3, and 7) are the middle row 
of the container and the positions of the $\square$ entries (1 and 6) are the bottom row of the container.
The unadorned entries (in positions 9, 10, 11, 14, 15, and 16) are the floating numbers.
The reader can verify the componentwise inequality $\code(\sigma) \leq (c'_1, \dots, c'_n)$.
\end{proof}

The shuffle $(c'_1, \dots, c'_n)$ constructed in the proof of Lemma~\ref{code-map-is-defined}
will be important in Section~\ref{Harmonic}, so we give it a name.

\begin{defn}
\label{maxcode-definition}
Let $\sigma \in \OP_{n,\lambda}$ for some partition 
$\lambda = (\lambda_1 \geq \cdots \geq \lambda_s)$ of $k$. 
 The shuffle $(c'_1, \dots, c'_n)$ of the sequences
 of the sequences
\begin{equation*}
(\lambda'_1 - 1, \lambda'_1 - 2, \dots, 1, 0), \dots, (\lambda'_k - 1, \lambda'_k - 2, \dots, 1, 0), \text{ and }
(\overbrace{s-1, s-1, \dots, s-1}^{n-k})
\end{equation*}
obtained by placing $(\lambda'_i - 1, \lambda'_i - 2, \dots, 1, 0)$ into the positions contained in the $i^{th}$
row from the top of the container of $\sigma$, and placing $s-1$ into all floating positions of $\sigma$,
will be referred to as $\maxcode(\sigma)$.
\end{defn}

For example, if 
$\sigma = ( 1, \, 3, \, 5, \, 9 \, \mid \, 6, \, 7, \, 8, \, 10, \, 14 \, \mid  \, 2, \, 12, \, 15 \, \mid  4, \, 13 \, \mid \, \varnothing \, \mid
\, 11, \, 16 ) \in \OP_{16,(3,3,2,2,0,0)}$ is the ordered set partition in the proof of Lemma~\ref{code-map-is-defined}
then
\begin{equation*}
\maxcode(\sigma) = (1,3,2,1,3,0,0,2,5,5,5,1,0,5,5,5).
\end{equation*}
The proof of Lemma~\ref{code-map-is-defined} gives the following result immediately.

\begin{lemma}
\label{code-less-than-max}
For any $\sigma \in \OP_{n,\lambda}$ we have the componentwise inequality $\code(\sigma) \leq \maxcode(\sigma)$.
\end{lemma}

By Lemma~\ref{code-map-is-defined}, we have a well-defined map
\begin{equation}
\code: \OP_{n,\lambda} \rightarrow \CCC_{n,\lambda}
\end{equation}
which sends $\sigma \in \OP_{n,\lambda}$ to its coinversion code $\code(\sigma) = (c_1, \dots, c_n)$.
Our first main result states that this map is a bijection.

\begin{theorem}
\label{code-is-bijection}
Let $k \leq n$ be positive integers and let $\lambda = (\lambda_1 \geq \cdots \geq \lambda_s) \vdash k$
be a partition of $k$. The map 
$\code: \OP_{n,\lambda} \rightarrow \CCC_{n,\lambda}$ is a bijection.
\end{theorem}

\begin{proof}
In order to prove that $\code$ is a bijection, we construct its inverse 
$\CCC_{n,\lambda} \rightarrow \OP_{n,\lambda}$.
Given $(c_1, \dots, c_n) \in \CCC_{n,\lambda}$, we define $\iota(c_1, \dots, c_n) \in \OP_{n,\lambda}$ by
the following insertion algorithm.

The element $\iota(c_1, \dots, c_n) \in \OP_{n,\lambda}$ 
will be constructed by starting with an empty container
of shape $\lambda$ and inserting the numbers $1, 2, \dots, n$ (in that order) to yield an element of 
$\OP_{n,\lambda}$. 
To describe what happens at a typical step of this insertion process, 
consider an ordered set partition $(B_1 \mid \cdots \mid B_s)$ with $s$ blocks.
We place the blocks $B_1, \dots, B_s$ in the container diagram corresponding to $\lambda$, from
left to right.
For example, if $\lambda = (3,3,2,2,0,0)$ and 
\begin{equation*}
(B_1 \mid \cdots \mid B_s) = ( 4 \, \mid \, 2, \, 3, \, 6 \, \mid \, 1 \, \mid \, \varnothing \, \mid \, \varnothing \, \mid \, 5)
\end{equation*}
our diagram is shown below.  In particular, the first, third, and fourth container columns from the left
remain unfilled.
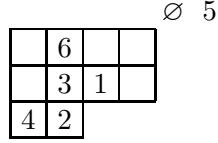
\begin{equation*}
\begin{Young}
  , & , & , & , &, \varnothing & , 5 \cr
  & 6  &  &  &, &,   \cr
  & 3 & 1 &  &, &, \cr
 4  & 2  &, &, &, &,  \cr 
, \cr
,1 & ,3 & ,2 & ,0 & ,4 & ,5 
\end{Young}
\end{equation*}
We label the blocks of $(B_1 \mid \cdots \mid B_s)$ (or equivalently the columns of its container diagram)
with the $s$ distinct {\em coinversion labels} $0, 1, 2, \dots, s-1$ according to the following rules
\begin{enumerate}
\item any unfilled container column receives a smaller coinversion label than any filled container column,
\item the coinversion labels of the filled container columns increase from left to right,
\item given two unfilled container columns with different numbers of empty boxes, the column with more empty 
boxes has a smaller coinversion label,
\item the coinversion labels of unfilled container columns with the same numbers of empty boxes increase from
right to left.
\end{enumerate}
The coinversion labels are displayed below the columns of the container diagram.

The element $\iota(c_1, \dots, c_n) \in \OP_{n,\lambda}$ is defined as follows.
Starting with the `empty' container diagram corresponding to $(\varnothing \mid \cdots \mid \varnothing)$,
for each $i = 1, 2, \dots, n$, we insert $i$ into the unique column with coinversion label $c_i$ 
(updating the coinversion labels as we go).

The map $\iota$ is best understood with an example. If $\lambda = (3,3,2,2,0,0)$ as above (so that 
$s = 6$) and $n = 16$, we have
\begin{equation*}
(c_1, \dots, c_{16}) = (1, 2, 2, 1, 3, 0, 0, 2, 2, 3, 5, 1, 0, 1, 2, 5) \in \CCC_{n,\lambda}.
\end{equation*}
The insertion procedure defining $\iota(c_1, \dots, c_{16})$ proceeds as follows.
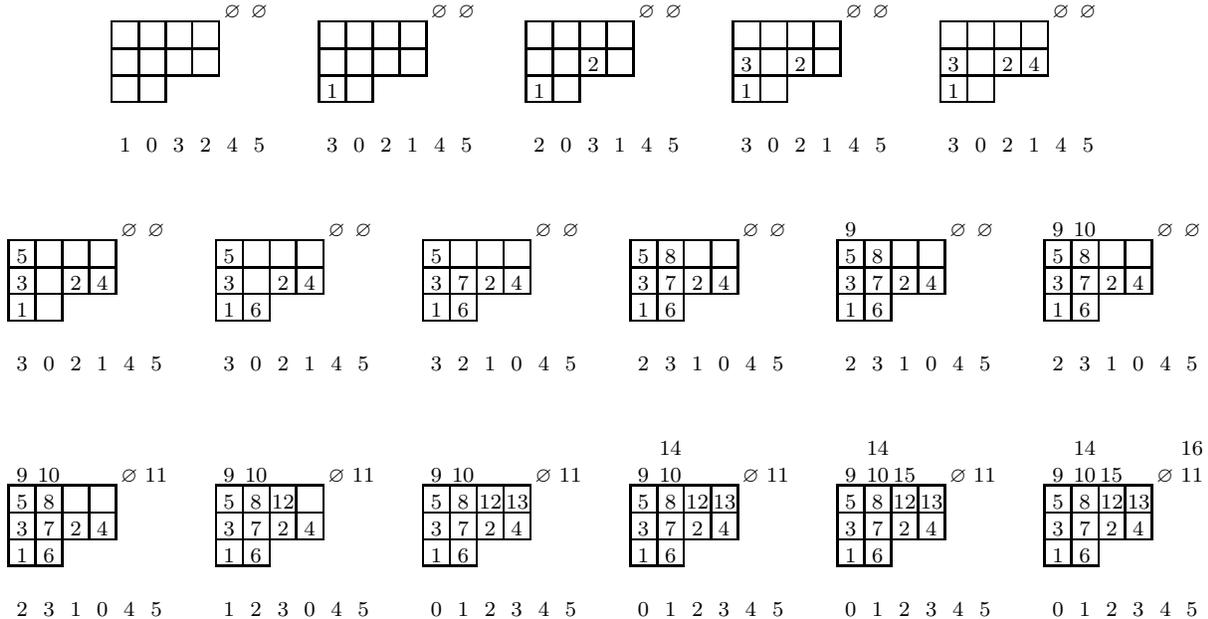
\begin{scriptsize}
\begin{equation*}
\begin{Young}
  , & , & , & , &, \varnothing & , \varnothing \cr
  &   &  &  &, &,   \cr
  &  &  &  &, &, \cr
   &   &, &, &, &,  \cr 
, \cr
,1 & ,0 & ,3 & ,2 & ,4 & ,5 
\end{Young}  \quad \quad 
\begin{Young}
  , & , & , & , &, \varnothing & , \varnothing \cr
  &   &  &  &, &,   \cr
  &  &  &  &, &, \cr
1   &   &, &, &, &,  \cr 
, \cr
,3 & ,0 & ,2 & ,1 & ,4 & ,5 
\end{Young}  \quad \quad
\begin{Young}
  , & , & , & , &, \varnothing & , \varnothing \cr
  &   &  &  &, &,   \cr
  &  & 2 &  &, &, \cr
1   &   &, &, &, &,  \cr 
, \cr
,2 & ,0 & ,3 & ,1 & ,4 & ,5 
\end{Young}  \quad \quad
\begin{Young}
  , & , & , & , &, \varnothing & , \varnothing \cr
  &   &  &  &, &,   \cr
3  &  & 2 &  &, &, \cr
1   &   &, &, &, &,  \cr 
, \cr
,3 & ,0 & ,2 & ,1 & ,4 & ,5 
\end{Young}  \quad \quad
\begin{Young}
  , & , & , & , &, \varnothing & , \varnothing \cr
  &   &  &  &, &,   \cr
3  &  & 2 & 4  &, &, \cr
1   &   &, &, &, &,  \cr 
, \cr
,3 & ,0 & ,2 & ,1 & ,4 & ,5 
\end{Young}  \quad \quad
\end{equation*}

\begin{equation*}
\begin{Young}
  , & , & , & , &, \varnothing & , \varnothing \cr
 5 &   &  &  &, &,   \cr
3  &  & 2 & 4  &, &, \cr
1   &   &, &, &, &,  \cr 
, \cr
,3 & ,0 & ,2 & ,1 & ,4 & ,5 
\end{Young}  \quad \quad
\begin{Young}
  , & , & , & , &, \varnothing & , \varnothing \cr
 5 &   &  &  &, &,   \cr
3  &  & 2 & 4  &, &, \cr
1   &  6 &, &, &, &,  \cr 
, \cr
,3 & ,0 & ,2 & ,1 & ,4 & ,5 
\end{Young}  \quad \quad
\begin{Young}
  , & , & , & , &, \varnothing & , \varnothing \cr
 5 &   &  &  &, &,   \cr
3  & 7 & 2 & 4  &, &, \cr
1   &  6 &, &, &, &,  \cr 
, \cr
,3 & ,2 & ,1 & ,0 & ,4 & ,5 
\end{Young}  \quad \quad
\begin{Young}
  , & , & , & , &, \varnothing & , \varnothing \cr
 5 & 8   &  &  &, &,   \cr
3  & 7 & 2 & 4  &, &, \cr
1   &  6 &, &, &, &,  \cr 
, \cr
,2 & ,3 & ,1 & ,0 & ,4 & ,5 
\end{Young}  \quad \quad
\begin{Young}
   , 9 & , & , & , &, \varnothing & , \varnothing \cr
 5 & 8   &  &  &, &,   \cr
3  & 7 & 2 & 4  &, &, \cr
1   &  6 &, &, &, &,  \cr 
, \cr
,2 & ,3 & ,1 & ,0 & ,4 & ,5 
\end{Young}  \quad \quad
\begin{Young}
   , 9 & , 10 & , & , &, \varnothing & , \varnothing \cr
 5 & 8   &  &  &, &,   \cr
3  & 7 & 2 & 4  &, &, \cr
1   &  6 &, &, &, &,  \cr 
, \cr
,2 & ,3 & ,1 & ,0 & ,4 & ,5 
\end{Young}  \quad \quad
\end{equation*}

\begin{equation*}
\begin{Young}
   , 9 & , 10 & , & , &, \varnothing & , 11 \cr
 5 & 8   &  &  &, &,   \cr
3  & 7 & 2 & 4  &, &, \cr
1   &  6 &, &, &, &,  \cr 
, \cr
,2 & ,3 & ,1 & ,0 & ,4 & ,5 
\end{Young}  \quad \quad
\begin{Young}
   , 9 & , 10 & , & , &, \varnothing & , 11 \cr
 5 & 8   &  12 &  &, &,   \cr
3  & 7 & 2 & 4  &, &, \cr
1   &  6 &, &, &, &,  \cr 
, \cr
,1 & ,2 & ,3 & ,0 & ,4 & ,5 
\end{Young}  \quad \quad
\begin{Young}
   , 9 & , 10 & , & , &, \varnothing & , 11 \cr
 5 & 8   &  12 & 13  &, &,   \cr
3  & 7 & 2 & 4  &, &, \cr
1   &  6 &, &, &, &,  \cr 
, \cr
,0 & ,1 & ,2 & ,3 & ,4 & ,5 
\end{Young}  \quad \quad
\begin{Young}
   , & , 14 &, &, &, &, \cr
   , 9 & , 10 & , & , &, \varnothing & , 11 \cr
 5 & 8   &  12 & 13  &, &,   \cr
3  & 7 & 2 & 4  &, &, \cr
1   &  6 &, &, &, &,  \cr 
, \cr
,0 & ,1 & ,2 & ,3 & ,4 & ,5 
\end{Young}  \quad \quad
\begin{Young}
   , & , 14 &, &, &, &, \cr
   , 9 & , 10 & , 15 & , &, \varnothing & , 11 \cr
 5 & 8   &  12 & 13  &, &,   \cr
3  & 7 & 2 & 4  &, &, \cr
1   &  6 &, &, &, &,  \cr 
, \cr
,0 & ,1 & ,2 & ,3 & ,4 & ,5 
\end{Young}  \quad \quad
\begin{Young}
   , & , 14 &, &, &, &, 16 \cr
   , 9 & , 10 & , 15 & , &, \varnothing & , 11 \cr
 5 & 8   &  12 & 13  &, &,   \cr
3  & 7 & 2 & 4  &, &, \cr
1   &  6 &, &, &, &,  \cr 
, \cr
,0 & ,1 & ,2 & ,3 & ,4 & ,5 
\end{Young}  \quad \quad
\end{equation*}
\end{scriptsize}
We conclude that $\iota(c_1, \dots, c_{16})$ is the element $\sigma \in \OP_{16,\lambda}$
displayed in \eqref{example-sigma}.

In order to verify that the map $\iota: \CCC_{n,\lambda} \rightarrow \OP_{n,\lambda}$ is well-defined,
we must show that the insertion procedure defining $\iota$ always fills every box in the container
corresponding to $\lambda$.
To do this, we induct on $n$.

Recall that $\ell(\lambda)$ is the number of nonzero parts in the partition 
$\lambda = (\lambda_1 \geq \cdots \geq \lambda_s)$ of $k$.  For $1 \leq i \leq \ell(\lambda)$, let $\lambda^{(i)}$ be
the partition obtained by sorting the sequence
$(\lambda_1, \dots, \lambda_i - 1, \dots \lambda_s)$ into weakly decreasing order.
The set $\CCC_{n,\lambda}$ satisfies the following disjoint union decomposition based on the
first entry of its sequences.

\begin{multline}
\label{disjoint-union-decomposition}
\CCC_{n,\lambda} = \bigsqcup_{i = 1}^{\ell(\lambda)}
\{ (i-1, c_2, \dots, c_n) \,:\, (c_2, \dots, c_n) \in \CCC_{n-1, \lambda^{(i)}} \} \, \, \, \sqcup \\
\bigsqcup_{j = \ell(\lambda) + 1}^{s}
\{ (j-1, c_2, \dots, c_n) \,:\, (c_2, \dots, c_n) \in \CCC_{n-1, \lambda} \} 
\end{multline}
Equation~\eqref{disjoint-union-decomposition}
is equivalent to a result of Griffin \cite[Lem. 3.8]{Griffin}.

Given $(c_1, c_2, \dots, c_n) \in \CCC_{n,\lambda}$, the algorithm $\iota$ starts by placing 
$1$ in the column with coinversion label $c_1$.  If $c_1 < \ell(\lambda)$, the entry $1$ fills a box in the container of 
$\lambda$, and the columns formed by the remaining container boxes 
(as well as their coinversion labels)
rearrange to give the container corresponding
to $\lambda^{(c_1 + 1)}$.  If $c_1 \geq \ell(\lambda)$, the entry $1$ is floating, and the container 
remains unchanged. 
Equation~\eqref{disjoint-union-decomposition} and induction on $n$ guarantee that 
the algorithm $\iota$ fills the container of $\lambda$, so that 
$\iota: \CCC_{n,\lambda} \rightarrow \OP_{n,\lambda}$ is well-defined.
It is routine to check that the maps $\code$ and $\iota$ are mutually inverse.
\end{proof}

\section{The harmonic space $V_{n,\lambda}$}
\label{Harmonic}

Throughout this section, we fix $k \leq n$ and let $\lambda = (\lambda_1 \geq \cdots \geq \lambda_s \geq 0)$
be a partition of $k$ (with trailing zeros allowed).
We write $\lambda'$ for the partition conjugate to $\lambda$.

\subsection{Injective tableaux and their polynomials}
Let $T \in \Inj(\lambda, \leq n)$ be an injective tableau of shape $\lambda$ with entries $\leq n$.
We introduce the monomial
$\xx(T) = x_1^{a_1} \cdots x_n^{a_n}$ where
\begin{equation}
a_i = \begin{cases}
b & \text{if $i$ appears in $T$ with $b$ boxes directly below it}, \\
s-1 & \text{if $i$ does not appear in $T$}.
\end{cases}
\end{equation}
As an example, if $\lambda = (3,3,1,0,0)$ (so that $s = 5$) and 
\begin{equation}
\label{example-tableau}
\begin{Young}
,& ,& 2 & 1 & 3 \cr
,T = &, & 5 & 4 & 9 \cr
,& ,& 6
\end{Young}
\end{equation}
we have
\begin{equation*}
\xx(T) = (x_2^2 x_5^1 x_6^0) \times (x_1^1 x_4^0) \times (x_3^1 x_9^0) \times (x_7^4 x_8^4) =  
x_1^1 x_2^2 x_3^1 x_4^0 x_5^1 x_6^0 x_7^4 x_8^4 x_9^0.
\end{equation*}

Given a tableau $T \in \Inj(\lambda, \leq n)$, we let 
$C_T \subseteq \symm_n$ be the parabolic subgroup of permutations $w \in \symm_n$
which stabilize the columns of $T$ and satisfy $w(i) = i$ for any $1 \leq i \leq n$ which does not appear in $T$.
In our case, we have $C_T = \symm_{\{2,5,6\}} \times \symm_{\{1,4\}} \times \symm_{\{3,9\}} \subseteq \symm_9$.
We also define the group algebra element
$\varepsilon_T \in \QQ[\symm_n]$ by
algebra element
\begin{equation}
\varepsilon_T := \sum_{w \in C_T} \sign(w) \cdot w.
\end{equation}

\begin{defn}
Let $T \in \Inj(\lambda, \leq n)$. The polynomial $\delta_T \in \QQ[\xx_n]$ is 
\begin{equation}
\delta_T := \varepsilon_T \cdot \xx(T),
\end{equation}
the image of $\xx(T)$ under the group algebra element $\varepsilon_T$.
\end{defn}

The notation $\delta_T$ is justified as follows.   If $T$ has columns $C_1, \dots, C_r$, then
$\delta_T$ factors as 
\begin{equation}
\delta_T = \delta_{C_1} \cdots \delta_{C_r} \times \prod_{\substack{1 \leq i \leq n \\ \text{$i$ not appearing in $T$}}} x_i^{s-1}
\end{equation}
where $\delta_{C_j}$ is the Vandermonde in the set of variables whose indices appear in $C_j$.
In our example we have
\begin{align*}
\delta_T &= \delta_{\{2,5,6\}} \times \delta_{\{1,4\}} \times \delta_{\{3,9\}} \times x_7^4 x_8^4 \\
&= (x_2 - x_5) (x_2 - x_6)(x_5 - x_6) \times 
(x_1 - x_4) \times (x_3 - x_9) \times x_7^4 x_8^4.
\end{align*}
The polynomial $\delta_T$ and the monomial $\xx(T)$ are related as follows.

\begin{observation}
\label{leading-delta}
The lexicographical leading term of $\delta_T$ is $\xx(T)$.
\end{observation}

\subsection{A generating set for $V_{n,\lambda}$ as a $\QQ[\xx_n]$-module}
Recall that $\QQ[\xx_n]$ acts on itself by the rule $f \odot g := (\partial f)(g)$.
The harmonic space $V_{n,\lambda}$ is a submodule for this action.
The polynomials $\delta_T$, where $T$ varies over $\Inj(\lambda, \leq n)$,
will turn out to generate the harmonic space $V_{n,\lambda}$ as a $\QQ[\xx_n]$-module.
We first establish that the $\delta_T$ are contained in $V_{n,\lambda}$.

\begin{lemma}
\label{delta-is-harmonic}
Let $T \in \Inj(\lambda, \leq n)$ be a tableau. The polynomial $\delta_T$ is contained in the harmonic space
$V_{n,\lambda}$.
\end{lemma}

\begin{proof}
It suffices to check that for each generator $f$ of the ideal $I_{n,\lambda}$ we have $f \odot \delta_T = 0$.
If $f = x_i^{s}$, the identity $f \odot \delta_T = 0$ follows from the fact that no exponents $\geq s$ appear
in $\xx(T)$ or in $\delta_T$.
We may therefore fix $1 \leq j \leq r$ and assume that $f$ is the elementary symmetric
polynomial $f = e_d(S)$ for some $S \subseteq [n]$ of size $|S| = n-j+1$
whose degree $d$ satisfies $d > |S| - \lambda'_{j} - \lambda'_{j+1} - \cdots - \lambda_n'$.
It suffices to show that $e_d(S) \odot \delta_T = 0$.
Without loss of generality we assume  $d \leq |S|$, so that $e_d(S) \neq 0$.

We give a combinatorial model for $e_d(S) \odot \delta_T$ as follows.  
Let $C_1, \dots, C_r$ be the columns of $T$, read from left to right.

The {\em $(S,T)$-staircase} consists of $n$ columns of boxes arranged as follows.
The $i^{th}$ column is decorated with the symbol $i_t$ where $t = 0$ if $i$ does not appear in $T$
and $i \in C_t$ otherwise.
If $i \notin S$, we further decorate $i_t$ with a circle $i_i^{\circ}$; such entries $i$ are called {\em frozen}.
If $i$ does not appear in $T$, the $i^{th}$ column of the $(S,T)$-staircase has $s-1$ boxes.
If $i$ is in row $r$ of $C_t$, the $i^{th}$ column has $|C_t| - r = \lambda'_t - r$ boxes.

Let us give an example of these concepts. Suppose $n = 9, s = 5, \lambda = (3,3,1)$, and $T$ is as 
in \eqref{example-tableau}.
Let $j = 2$ and take $S = \{1,2,3,4,6,7,8,9 \}$ so that 
$[n] - S = \{5\}$.
The $(S,T)$-staircase is as follows.

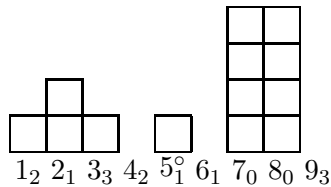
\begin{equation*}
\begin{young}
, & ,&, &, & ,&, & & &, \cr
 , & ,&,&, & ,&, & & &, \cr
 ,  & &, &, & ,&, & & &, \cr
    & & &, & &, & & &, \cr
,1_2 & ,2_1  & ,
3_3 & ,4_2 & ,5_1^{\circ} & ,6_1 & ,7_0 & ,8_0 & ,9_3
\end{young}
\end{equation*}

We apply permutations $w \in C_T$ to sequences $(B_1, \dots, B_n)$ of $n$ stacks of 
boxes by rearranging the box stacks.
A {\em permuted $(S,T)$-staircase} $\sigma$ is obtained from the $(S,T)$-staircase by applying some permutation
$w \in C_T$ which stabilizes the columns of $T$.
If $(S,T)$ is as above and $w = (2,5,6) (3,9) \in C_T$, the associated permuted staircase $\sigma$ is

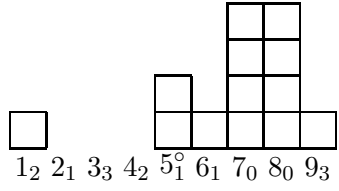
\begin{equation*}
\begin{young}
, & ,&, &, & ,&, & & &, \cr
 , & ,&,&, & ,&, & & &, \cr
 ,  &, &, &, & &, & & &, \cr
    & ,& ,&, & & & & & \cr
,1_2 & ,2_1  & ,
3_3 & ,4_2 & ,5_1^{\circ} & ,6_1 & ,7_0 & ,8_0 & ,9_3
\end{young}
\end{equation*}
Observe that $w$ leaves the labels unchanged.
It should be clear that, for fixed $S$ and $T$, a permuted $(S,T)$-staircase determines
the permutation $w \in C_T$ uniquely.

The {\em sign} of an $(S,T)$-staircase $\sigma$ is the sign of the permutation $w \in C_T$,
i.e. $\mathrm{sign}(\sigma) = \sign(w)$;
in our example  $\sign((2,5,6)(3,9)) = -1$.
The {\em weight} $\wt(\sigma)$ of $\sigma$ is the monomial $\wt(\sigma) = x_1^{a_1} \cdots x_n^{a_n}$,
where $a_i$ is the number of boxes in column $i$. 
In our example $\wt(\sigma) = x_1 x_5^2 x_6 x_7^4 x_8^4 x_9$.
The polynomial $\delta_T$ has the combinatorial interpretation
\begin{equation}
\delta_T = \sum_{\sigma} \sign(\sigma) \cdot \wt(\sigma),
\end{equation}
where the sum is over all permuted $(S,T)$-staircases $\sigma$.

A {\em $d$-dotted permuted $(S,T)$-staircase} $\sigma^{\bullet}$
is obtained from a permuted $(S,T)$-staircase $\sigma$ by marking
 $d$ boxes with $\bullet$ so that no two marked boxes are in the same column and 
 so that no frozen column $i_t^{\circ}$ gets a marked box.
 With $d = 4$ and $\sigma$  as above, an example choice for
 $\sigma^{\bullet}$ is as follows.
 
 
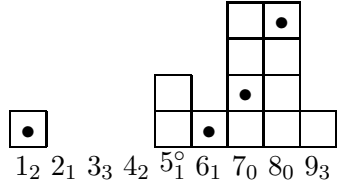
\begin{equation*}
\begin{young}
, & ,&, &, & ,&, & & \bullet &, \cr
 , & ,&,&, & ,&, &  & &, \cr
 ,  &, &, &, & &, &\bullet & &, \cr
   \bullet & ,& ,&, & & \bullet & & &  \cr
,1_2 & ,2_1  & ,
3_3 & ,4_2 & ,5_1^{\circ} & ,6_1 & ,7_0 & ,8_0 & ,9_3
\end{young}
\end{equation*}
The {\em sign} of $\sigma^{\bullet}$ is the same as the sign of the unmarked $(S,T)$-staircase $\sigma$, i.e.
$\sign(\sigma^{\bullet}) = \sign(\sigma)$. The {\em weight} $\wt(\sigma^{\bullet})$ is the monomial 
$x_1^{a_1} \cdots x_n^{a_n}$ where $a_i$ is the number of unmarked boxes in column $i$; in the above 
example
$\wt(\sigma^{\bullet}) = x_5^2 x_7^3 x_8^3 x_9$.

The polynomial $e_d(S) \odot \delta_T$ has a combinatorial interpretation in terms of dotted
permuted staircases. More precisely 
we have 
\begin{equation}
\label{combinatorial-interpretation}
e_d(S) \odot \delta_T = \sum_{\sigma^{\bullet}} \sign(\sigma^{\bullet}) \cdot \wt(\sigma^{\bullet}),
\end{equation}
where the sum is over all $d$-dotted permuted $(S,T)$-staircases $\sigma^{\bullet}$.
Our goal is to show that Equation~\eqref{combinatorial-interpretation} equals zero.

We use a sign-reversing involution to prove that the right-hand side of 
Equation~\eqref{combinatorial-interpretation} vanishes.
The following key observation may be verified from our assumptions on $|S|$ and $d$.

{\bf Observation:}
{\em For any $d$-dotted permuted $(S,T)$-staircase $\sigma^{\bullet}$, there is some value
$1 \leq t \leq r$ such that for the corresponding column $C_t$ of $T$:
\begin{enumerate}
\item
 no column of $\sigma^{\bullet}$ 
indexed by an entry in $C_t$  is frozen,  and 
\item at least
one column of $\sigma^{\bullet}$ indexed by an entry of $C_t$ contains a $\bullet$.
\end{enumerate}}

In our running example, we may take $t = 2$, so that the corresponding column $C_2$ of $T$
has entries $1$ and $4$. Neither column 1 nor column 4 of $\sigma^{\bullet}$ is frozen and 
column 1 contains a $\bullet$.

If $\sigma^{\bullet}$ is a $d$-dotted permuted $(S,T)$-staircase, let $t \geq 1$ be minimal such
 that $t$ is as in the above observation.
 Two of the columns in $\sigma^{\bullet}$ indexed  by entries in
 $C_t$ must contain the same number of unmarked boxes.
 Let $\iota(\sigma^{\bullet})$ be obtained from $\sigma^{\bullet}$ by interchanging the two such columns
 of minimal height. In our running example,  we have $t = 2$ and $\iota(\sigma^{\bullet})$ is given by
 interchanging columns 1 and 4:
  
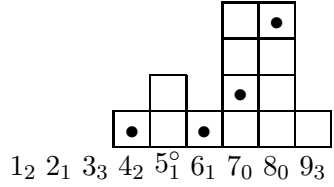
\begin{equation*}
\begin{young}
, & ,&, &, & ,&, & & \bullet &, \cr
 , & ,&,&, & ,&, &  & &, \cr
 ,  &, &, &, & &, &\bullet & &, \cr
  , & ,& ,& \bullet & & \bullet & & &  \cr
,1_2 & ,2_1  & ,
3_3 & ,4_2 & ,5_1^{\circ} & ,6_1 & ,7_0 & ,8_0 & ,9_3
\end{young}
\end{equation*}
For any $d$-dotted permuted $(S,T)$-staircase $\sigma^{\bullet}$ we have
\begin{equation}
\iota(\iota(\sigma^{\bullet})) = \sigma^{\bullet}, \quad
\sign(\iota(\sigma^{\bullet})) = - \sign(\sigma^{\bullet}), \quad \text{and} \quad
\wt(\iota(\sigma^{\bullet})) = \wt(\sigma^{\bullet}).
\end{equation}
That is, the map $\sigma^{\bullet} \mapsto \iota(\sigma^{\bullet})$ is a weight-preserving and sign-reversing
involution which verifies that the right-hand side of Equation~\eqref{combinatorial-interpretation} vanishes.
\end{proof}

\begin{defn}
\label{tableau-partition}
Let $k \leq n$, let $\lambda = (\lambda_1 \geq \cdots \geq \lambda_s)$ be a partition of $k$, and let 
$\sigma \in \OP_{n,\lambda}$. Define $T(\sigma) \in \Inj(\lambda, \leq n)$ to be the tableau whose 
$i^{th}$ column consists in the entries in row $i$ from the top of the container of $\sigma$.
\end{defn}

For example, if $\sigma$ is as in \eqref{example-sigma}, then $T(\sigma)$ is shown below.
\begin{equation*}
\begin{Young}
 5 & 2 & 1 \cr
 8 & 3 & 6 \cr
12 & 4 \cr
13 & 7
\end{Young}
\end{equation*}
In particular, floating entries in $\sigma$ do not appear in $T(\sigma)$.  The most important 
property of $T(\sigma)$ is shown below; it follows from Observation~\ref{leading-delta}
and the definition of $\maxcode$.

\begin{observation}
\label{maxcode-leading}
The lexicographical leading term of $\delta_{T(\sigma)}$ is the monomial
$x_1^{a_1} \cdots x_n^{a_n}$ where $(a_1, \dots, a_n) = \maxcode(\sigma)$.
\end{observation}

Observation~\ref{maxcode-leading} allows us to associate a harmonic polynomial to any 
ordered set partition $\sigma \in \OP_{n,\lambda}$.

\begin{defn}
\label{partition-polynomial}
Let $k \leq n$, let $\lambda = (\lambda_1 \geq \cdots \geq \lambda_s)$ be a partition of $k$, and let 
$\sigma \in \OP_{n,\lambda}$ satisfy $\code(\sigma) = (c_1, \dots, c_n)$ and
$\maxcode(\sigma) = (a_1, \dots, a_n)$.  Define a polynomial $\delta_{\sigma} \in \QQ[\xx_n]$ by
\begin{equation}
\delta_{\sigma} := (x_1^{a_1 - c_1} \cdots x_n^{a_n - c_n}) \odot \delta_{T(\sigma)}.
\end{equation}
\end{defn}

We are ready to describe our generating set for the harmonic space $V_{n,\lambda}$ 
as a $\QQ[\xx_n]$-module.

\begin{theorem}
\label{generating-harmonic-set}
Let $k \leq n$ be positive integers and let $\lambda = (\lambda_1 \geq \cdots \geq \lambda_s)$ be a partition
of $k$.  The harmonic space $V_{n,\lambda}$ is the smallest subspace of $\QQ[\xx_n]$ which
\begin{itemize}
\item contains the polynomial $\delta_T$ for every tableau $T \in \Inj(\lambda, \leq n)$, and
\item is closed under the partial derivative operators
$\partial/\partial x_1, \dots, \partial/\partial x_n$.
\end{itemize}
Equivalently, the set $\{ \delta_T \,:\, T \in \Inj(\lambda, \leq n) \}$ generates $V_{n,\lambda}$ 
as a $\QQ[\xx_n]$-module.
\end{theorem}

\begin{proof}
Let $W_{n,\lambda}$ be the subspace defined by the two bullet points of the theorem.
By Lemma~\ref{delta-is-harmonic} we have the containment of vector spaces
\begin{equation} 
W_{n,\lambda} \subseteq V_{n,\lambda}.
\end{equation}
We also have 
\begin{equation}
\dim V_{n,\lambda} = |\CCC_{n,\lambda}| = |\OP_{n,\lambda}|,
\end{equation}
where the first equality follows from Theorem~\ref{griffin-basis-theorem} and the second
follows from Theorem~\ref{code-is-bijection}.  It therefore suffices to exhibit 
$|\OP_{n,\lambda}|$ linearly independent elements of $W_{n,\lambda}$.

Indeed, Observation~\ref{maxcode-leading} and Definition~\ref{partition-polynomial} imply that 
the lexicographical leading term of $\delta_{\sigma}$ has exponent sequence given by
$\code(\sigma)$.  Theorem~\ref{code-is-bijection} guarantees that the set
$\{ \delta_{\sigma} \,:\, \sigma \in \OP_{n,\lambda} \}$ is linearly independent 
and Lemma~\ref{delta-is-harmonic} assures that the polynomials in this set lie in $W_{n,\lambda}$.
\end{proof}

The proof of Theorem~\ref{generating-harmonic-set} also yields a harmonic basis of $R_{n,\lambda}$.

\begin{theorem}
\label{harmonic-basis}
Let $k \leq n$ be positive integers and let $\lambda = (\lambda_1 \geq \cdots \geq \lambda_s)$ be a partition
of $k$.  The set 
\begin{equation}
\{ \delta_{\sigma} \,:\, \sigma \in \OP_{n,\lambda} \}
\end{equation}
is a harmonic basis of $R_{n,\lambda}$.
\end{theorem}

Recall that the {\em Hilbert series} of a graded $\QQ$-algebra $R = \bigoplus_{d \geq 0} R_d$ with each
graded piece finite-dimensional is the formal power series
\begin{equation}
\Hilb(R; q) := \sum_{d \geq 0} (\dim R_d) \cdot q^d.
\end{equation}
We have a combinatorial expression for the Hilbert series of $R_{n,\lambda}$.

\begin{corollary}
\label{hilbert-series}
Let $k \leq n$ be positive integers and let $\lambda = (\lambda_1 \geq \cdots \geq \lambda_s)$ be a partition
of $k$. We have
\begin{equation}
\Hilb(R_{n,\lambda}; q) = \sum_{\sigma \in \OP_{n,\lambda}} q^{\coinv(\sigma)}.
\end{equation}
\end{corollary}

\begin{proof}
The degree of the polynomial $\delta_{\sigma}$ is $\coinv(\sigma)$. Now apply
Theorem~\ref{harmonic-basis}.
Alternatively, combine Theorems~\ref{griffin-basis-theorem} and 
\ref{code-is-bijection}.
\end{proof}

\begin{remark}
Griffin found \cite[Thm. 5.12]{Griffin} a variant of the $\coinv$ statistic which encodes 
the graded $\symm_n$-isomorphism type of $R_{n,\lambda}$. 
This may be viewed as a moral extension of Corollary~\ref{hilbert-series}.
Recall that the {\em graded Frobenius image} $\grFrob(V;q)$ of a graded $\symm_n$-module
$V = \bigoplus_{d \geq 0} V_d$ is the symmetric function 
\begin{equation}
\grFrob(V;q) = \sum_{d \geq 0} q^d \sum_{\rho \vdash n} 
\text{(multiplicity of the $\symm_n$-irreducible $S^{\rho}$ in $V_d$)} \cdot s_{\rho}(\xx)
\end{equation}
where $s_{\rho}(\xx)$ is the Schur function corresponding to the partition $\rho \vdash n$.

Given $n$ and $\lambda = (\lambda_1 \geq \cdots \geq \lambda_s)$ with $\lambda_s \geq 0$,
Griffin considers the family $\mathrm{ECI}_{n,\lambda}$ of 
{\em extended column-increasing fillings}.
These are sequences $\mu = (M_1 \mid \cdots \mid M_s)$ of $s$ multisets of positive integers
in which $|M_i| \geq \lambda_i$ for all $i$ such that $|M_1| + \cdots + |M_s| = n$.
Elements of $\mathrm{ECI}_{n,\lambda}$ may be viewed as fillings of the container diagram of $\lambda$
which weakly increase going up columns and fill every box in the container; as an example,
the element 
\begin{equation*}
\mu = ( 1, \, 1, \, 2, \, 4  \mid 2, \, 4, \, 4, \, 5, \, 6 \mid 1, \, 2, \, 4 \mid 3, \, 3 \mid \varnothing \mid 1, \, 1) \in
\mathrm{ECI}_{16,(3,3,2,2,0,0),6}
\end{equation*}
is shown below.

\begin{equation*}
\label{example-mu}
\begin{Young}
 ,& ,6 & ,&,  &,  &,1  \cr
 ,4 & ,5 & ,4 &, &, \varnothing  &, 1 \cr
 2 & 4 & 2 & 3 &, &,   \cr
 1 &  4 & 1 & 3 &, &, \cr
1   & 2 &, &, &, &,  \cr 
\end{Young}
\end{equation*}

Given $\mu \in \mathrm{ECI}_{n,\lambda}$, we attach the monomial $\xx^{\mu}$ in the variable 
set $\xx = (x_1, x_2, \dots )$ where the exponent of $x_i$ is the multiplicity of $i$ in $\mu$.
For $\mu$ as above, we have $\xx^{\mu} = x_1^5 x_2^3 x_3^2 x_4^4 x_5 x_6$.
Griffin defines a statistic $\mathrm{inv}$ on $\mathrm{ECI}_{n,\lambda}$ which is similar to
(but not quite the same as) our statistic $\coinv$ -- see \cite[Rmk. 5.3]{Griffin} -- and proves that the 
graded Frobenius image of $R_{n,\lambda}$ is given combinatorially by
\begin{equation}
\grFrob(R_{n,\lambda};q) = \sum_{\mu \in \mathrm{ECI}_{n,\lambda}} 
q^{\mathrm{inv}(\mu)} \xx^{\mu}.
\end{equation}
\end{remark}

\subsection{A conjecture of Wilson}
For positive integers $n, k,$ and $s$, the following ideal $I_{n,k,s} \subseteq \QQ[\xx_n]$ was introduced
in \cite{HRS}:
\begin{equation}
I_{n,k,s} := \langle e_n(\xx_n), e_{n-1}(\xx_n), \dots, e_{n-k+1}(\xx_n), x_1^s, x_2^s, \dots, x_n^s \rangle.
\end{equation}
Let $R_{n,k,s} := \QQ[\xx_n]/I_{n,k,s}$ be the corresponding quotient ring.

When $k \leq s$, we have $R_{n,k,s} = R_{n,\lambda}$ where 
$\lambda = (1^k, 0^{s-k})$ and the structure of $R_{n,k,s}$ as a graded $\symm_n$-module was determined
by Haglund-Rhoades-Shimozono \cite{HRS}.
Indeed, the ring $R_{n,k,s}$ was one of the motivating examples for defining $R_{n,\lambda}$ 
for general $\lambda$.

When $k > s$, the structure of $R_{n,k,s}$ was not studied in \cite{HRS} because the Gr\"obner theory 
of the ideal $I_{n,k,s}$ was more complicated in this case.
Although it is not immediately obvious, we will establish that $R_{n,k,s}$ is also an instance of the 
$R_{n,\lambda}$ rings when $k > s$.

For positive integers $k$ and $s$, write
$k = q s + r$
for  integers $q, r \geq 0$ with $r < k$.  We define $\lambda(k,s) := ((q+1)^r, q^{s-r})$ 
to be the partition of $k$ given by $r$ copies of $q+1$ followed by $s-r$ copies of $q$.

\begin{proposition}
\label{quotient-remainder}
For any positive integers $n, k, s$ we have
$I_{n,k,s} = I_{n,\lambda}$ and $R_{n,k,s} = R_{n,\lambda}$ where $\lambda = \lambda(k,s)$.
\end{proposition}

Proposition~\ref{quotient-remainder} implies that 
$\dim R_{n,k,s} = |\OP_{n,\lambda(k,s)}|$.
This combinatorial expression for $\dim R_{n,k,s}$ was conjectured by Andy Wilson
(personal communication).

\begin{proof}
Let $\lambda =  \lambda(k,s)$.  Every generator of $I_{n,k,s}$ is also a generator of $I_{n,\lambda}$,
so we have the containment of ideals
\begin{equation}
I_{n,k,s} \subseteq I_{n,\lambda}.
\end{equation}
It therefore suffices to show that every generator of $I_{n,\lambda}$ lies in $I_{n,k,s}$.
The generators $x_i^s$ of $I_{n,\lambda}$ are also generators of $I_{n,k,s}$, so it suffices to 
check that the elementary symmetric polynomials $e_d(S)$ in partial variable sets $S \subseteq [n]$
which appear as generators of $I_{n,\lambda}$ lie in $I_{n,k,s}$.

Let $S \subseteq [n]$ and $1 \leq d \leq n$ be such that $e_d(S)$ is a generator of $I_{n,\lambda}$.
We prove that $e_d(S) \in I_{n,k,s}$ by descending induction on the size $|S|$ of $S$.
If $|S| = n$, we have $e_d(S) = e_d(\xx_n)$ and $d > n-k$, so that $e_d(S)$ is a generator of
$I_{n,k,s}$. 

If $|S| < n$, choose an arbitrary index $i \in [n] - S$ and let $T = S \cup \{i\}$.
The polynomial $e_d(S)$ may be expressed as 
\begin{multline}
\label{crux-of-wilson-proof}
e_d(S) = e_d(T) - x_i e_{d-1}(S) 
= e_d(T) - x_i e_{d-1}(T) + x_i^2 e_{d-2}(S)  \\ = \cdots =
\left[ \sum_{j = 0}^{s-1} (-1)^j x_i^j e_{d-j}(T) \right] + (-1)^s x_i^s e_{d-s}(S).
\end{multline}
We claim that each term in the sum on the right-hand side of Equation~\eqref{crux-of-wilson-proof}
lies in $I_{n,k,s}$, so that $e_d(S) \in I_{n,k,s}$.
Indeed, since $e_d(S)$ is a generator of $I_{n,\lambda}$, each column aside from the shortest column
of $\lambda$ has $s$ boxes,
and $T$ has one more element than $S$, each of the $s$ polynomials
$e_d(T), e_{d-1}(T), \dots, e_{d-s+1}(T)$ are  generators of $I_{n,\lambda}$, and so lie in $I_{n,k,s}$ by induction.
In particular, the sum in the square brackets lies in $I_{n,k,s}$. Furthermore,
the monomial $x_i^s$ is a generator of $I_{n,k,s}$, so that $(-1)^s x_i^s e_{d-s}(S) \in I_{n,k,s}$.
We conclude that $e_d(S) \in I_{n,k,s},$ finishing the proof.
\end{proof}

\subsection{The lexicographical leading terms of harmonic polynomials}
Let $<$ be the {\em lexicographical order} 
on monomials in $\QQ[\xx_n]$. That is, we have 
$x_1^{a_1} \cdots x_n^{a_n} < x_1^{b_1} \cdots x_n^{b_n}$ if and only if there is some $1 \leq i \leq n$
such that $a_1 = b_1, \dots, a_{i-1} = b_{i-1},$ and $a_i < b_i$.
It is well known that the lexicographical order $<$ is a {\em monomial order}, meaning that
\begin{itemize}
\item we have $1 \leq m$ for any monomial $m$ in $x_1, \dots, x_n$ and
\item given three monomials $m, m', m''$ with $m \leq m'$, we have $m \cdot m'' \leq m' \cdot m''$.
\end{itemize}
If $f \in \QQ[\xx_n]$ is a nonzero polynomial, let $\initial_< (f)$ denote the largest monomial in lexicographical
order appearing in $f$.

The coinversion codes $\CCC_{n,\lambda}$ of ordered 
set partitions in $\OP_{n,\lambda}$ are precisely the exponent sequences of the lexicographical
leading monomials of nonzero polynomials in $V_{n,\lambda}$.
This gives another connection between harmonic polynomials and ordered set partitions.

\begin{theorem}
\label{leading-term-theorem}
Let $k \leq n$ be positive integers and let $\lambda = (\lambda_1 \geq \cdots \geq \lambda_s)$ be a partition
of $k$ into $s$ nonnegative parts.  If $f \in V_{n,\lambda}$ is any nonzero harmonic polynomial with
$\initial_< f = x_1^{c_1} x_2^{c_2} \cdots x_n^{c_n}$ then 
$(c_1, c_2, \dots, c_n) \in \CCC_{n,\lambda}$.
\end{theorem}

\begin{proof}
The proof is by induction on $n$. If $n = 1$, then $\lambda = (0^s)$ or $\lambda = (1,0^{s-1})$.
If $\lambda = (0^s)$ then $I_{1, (0^s)} = \langle x_1^s \rangle \subseteq \QQ[x_1]$ so that 
$V_{1,(0^s)} = \mathrm{span}_{\QQ} \{ 1, x_1, x_1^2, \dots, x_1^{s-1} \}$.
Since $\CCC_{1,(0^s)} = \{ (0), (1), (2), \dots, (s-1) \}$, the result is true in this case.
If $\lambda = (1, 0^{s-1})$ then $I_{1,(1,0^{s-1})} = \langle x_1 \rangle \subseteq \QQ[x_1]$ so that
$V_{1,(1,0^{s-1})} = \QQ$. Since
$\CCC_{1,(1,0^{s-1})} = \{ (0) \}$, this completes the proof when $n = 1$.

When $n$ is arbitrary and $\lambda = (0^s)$ is a partition of $k = 0$, we may compute directly that
\begin{equation*}
V_{n,(0^s)} = \mathrm{span}_{\QQ} \{ x_1^{a_1} \cdots x_n^{a_n} \,:\, a_1, \dots, a_n < s \},
\end{equation*}
from which the theorem follows in this case.
We therefore assume going forward that $n, k > 0$.

Our main tool will be Griffin's disjoint union decomposition \eqref{disjoint-union-decomposition}
of $\CCC_{n,\lambda}$ according to the first terms of its sequences, recapitulated 
here for convenience. Recall that $\ell(\lambda)$ is the number of nonzero parts of $\lambda$.
\begin{multline}
\label{disjoint-union-decomposition-two}
\CCC_{n,\lambda} = \bigsqcup_{i = 1}^{\ell(\lambda)}
\{ (i-1, c_2, \dots, c_n) \,:\, (c_2, \dots, c_n) \in \CCC_{n-1, \lambda^{(i)}} \} \, \, \, \sqcup \\
\bigsqcup_{j = \ell(\lambda) + 1}^{s}
\{ (j-1, c_2, \dots, c_n) \,:\, (c_2, \dots, c_n) \in \CCC_{n-1, \lambda} \} 
\end{multline}

Let $f \in V_{n,\lambda}$ be a nonzero polynomial.  Since $f$ is harmonic and $x_1^s \in V_{n,\lambda}$,
we  have
\begin{equation}
x_1^s \odot f = \partial^s f/\partial x_1^s= 0.
\end{equation}
In particular, if $\initial_<(f) = x_1^{c_1} x_2^{c_2} \cdots x_n^{c_n}$ then $c_1 < s$ and, by the definition 
of $<$, no monomial appearing in $f$ has an exponent of $x_1$ greater than $c_1$.
In particular, the polynomial $x_1^{c_1} \odot f$ does not involve the variable $x_1$.
Since the lexicographical order is a monomial order, we have
\begin{equation}
\label{initial-monomial-relation}
\initial_< (f) = x_1^{c_1} \cdot \initial_< (x_1^{c_1} \odot f).
\end{equation}

Let  $W_{n-1,\lambda^{(c_1+1)}} \subseteq \QQ[x_2, x_3, \dots, x_n]$ 
be the image of the subspace $V_{n-1,\lambda^{(c_1+1)}} \subseteq \QQ[x_1, x_2, \dots, x_{n-1}]$
under the algebra map $x_i \mapsto x_{i+1}$. 
Similarly, let $W_{n-1,\lambda}$ be the image of $V_{n-1,\lambda}$ under $x_i \mapsto x_{i+1}$.
Thanks to \eqref{disjoint-union-decomposition-two},
\eqref{initial-monomial-relation}, and induction, it suffices to prove the following
claim. 

{\bf Claim:} {\em If $c_1 < \ell(\lambda)$ then $x_1^{c_1} \odot f \in W_{n-1, \lambda^{(c_1+1)}}$. 
If $\ell(\lambda) \leq c_1 < s$ then $x_1^{c_1} \odot f \in W_{n-1,\lambda}$.}

Let $J_{n-1,\lambda^{(c_1)}}, J_{n-1,\lambda} \subseteq \QQ[x_2, x_3, \dots, x_n]$ be the images of 
$I_{n-1, \lambda^{(c_1)}}, I_{n,\lambda} \subseteq \QQ[x_1, x_2, \dots, x_{n-1}]$
under the algebra map $x_i \mapsto x_{i+1}$.
We verify that $x_1^{c_1} \odot f$ is annihilated by the generators of the relevant $J$-ideal.
For the generators of the form $x_i^s$ this is clear, so we need only verify this statement for generators of 
the form $e_d(S)$. 

The proof of our Claim hinges on the following observation.

{\bf Observation:} {\em Let $S \subseteq \{2, 3, \dots, n \}$ and write
$S' := S \cup \{1\}$.
Let $d \geq 1$ be such that 
$e_d(S)$ appears as a generator of $J_{n-1,\lambda^{(c_1+1)}}$ (in the case $c_1 < \ell(\lambda)$)
or $J_{n-1, \lambda}$ (in the case $\ell(\lambda) \leq c_1 < s$).
Then $e_{d+1}(S')$ appears as a generator of $I_{n,\lambda}$ so that 
$e_{d+1}(S') \odot f = 0$.}

Let $d \geq 1$ and $S \subseteq \{2, 3, \dots, n\}$ be such that $e_d(S)$ is a 
nonzero generator of the $J$-ideal
described in the Observation.
Our analysis breaks up into cases depending on the value of $c_1$.

{\bf Case 1:} {\em We have $c_1 > 0$.}

In this case we compute
\begin{multline}
\label{case-one-induction}
e_d(S) \odot (x_1^{c_1} \odot f) = (x_1^{c_1} e_d(S)) \odot f = ( x_1^{c_1 - 1} e_{d+1}(S') - 
x_1^{c_1 - 1} e_{d+1}(S)) \odot f \\
= x_1^{c_1 - 1} e_{d+1}(S')  \odot f - x_1^{c_1 - 1} e_{d+1}(S) \odot f .
\end{multline}
By our Observation, we have $e_{d+1}(S') \odot f = 0$, so the first term on the right-hand side of
Equation~\eqref{case-one-induction} vanishes and we have
\begin{equation}
\label{case-one-induction-two}
e_d(S) \odot (x_1^{c_1} \odot f) = - x_1^{c_1 - 1} e_{d+1}(S) \odot f .
\end{equation}
We may now show $e_d(S) \odot (x_1^{c_1} \odot f) = 0$ by descending induction on $d$.
In the base case $d = |S|$, we have $e_{d+1}(S) = 0$ so this follows from
Equation~\eqref{case-one-induction-two}.
When $d < |S|$, the polynomial $e_{d+1}(S)$ is a nonzero generator of the $J$-ideal described in
the Observation, so that $e_{d+1}(S) \odot f = 0$ by induction, so 
Equation~\eqref{case-one-induction-two} shows $e_d(S) \odot (x_1^{c_1} \odot f) = 0$.

{\bf Case 2:} {\em We have $c_1 = 0$.}

Here is we make use of our assumption that $\lambda$ is a partition of a {\em positive} integer $k$ 
so that $\lambda_{c_1 + 1} = \lambda_1 > 0$.
We compute 
\begin{equation}
\label{case-two-more-direct}
e_d(S) \odot (x_1^{c_1} \odot f) = e_d(S) \odot f = e_d(S') \odot f.
\end{equation}
The second equality is true because $f$ does not involve the variable $x_1$ so that 
every monomial appearing in $e_d(S')$ involving $x_1$ annihilates $f$.
Since $\lambda^{(c_1 + 1)} = \lambda^{(1)}$ is the non-decreasing rearrangement of
$(\lambda_1 - 1, \lambda_2, \dots, \lambda_s)$
and $e_d(S)$ is a nonzero generator of $J_{n-1,\lambda^{(1)}}$, we see that 
$e_d(S')$ is a nonzero generator of $I_{n,\lambda}$ so that 
the right-hand side of Equation~\eqref{case-two-more-direct} vanishes.
\end{proof}

Theorem~\ref{leading-term-theorem} directly gives the upper bound
\begin{equation}
\label{harmonic-dimension-upper-bound}
\dim V_{n,\lambda} \leq | \CCC_{n,\lambda} |
\end{equation}
on the dimension of the harmonic space $V_{n,\lambda}$.
Indeed, if $N > | \CCC_{n,\lambda} |$ and we are given $N$ harmonic polynomials 
$f_1, f_2, \dots, f_N \in V_{n,\lambda}$, there exist $c_1, c_2, \dots, c_N \in \QQ$
not all zero so that 
\begin{quote}
for any monomial $m$ whose exponent sequence appears in $\CCC_{n,\lambda}$, the coefficient of 
$m$ in $f := c_1 f_1 + c_2 f_2 + \dots + c_N f_N$ is zero.
\end{quote}
But  $f \in V_{n,\lambda}$ is harmonic, so  Theorem~\ref{leading-term-theorem} forces $f = 0$,
implying that $f_1, f_2, \dots, f_N$ are linearly dependent.
Of course, Equation~\eqref{harmonic-dimension-upper-bound} 
also follows from Griffin's \cite{Griffin} result $\dim R_{n,\lambda}  = |\CCC_{n,\lambda}|$ and 
the vector space isomorphism between the quotient ring and harmonic space of a given ideal.

\section{Conclusion}
\label{Conclusion}

Let $k \leq n$ be positive integers and let $\lambda \vdash k$.
In this paper we studied the harmonic space $V_{n,\lambda}$ associated to 
the ring $R_{n,\lambda}$ using polynomials $\delta_T$ and $\delta_{\sigma}$ closely related to the Vandermonde
determinant $\delta_n \in \QQ[\xx_n]$.
In the case $\lambda_1 \leq 1$,
Rhoades and Wilson \cite{RW} gave an alternate harmonic-like model for $R_{n,\lambda}$ 
involving anticommuting variables as follows.

{\em Superspace} of rank $n$ is the $\QQ$-algebra  $\Omega_n$ given by a tensor product
\begin{equation}
\Omega_n := \QQ[x_1, \dots, x_n] \otimes \wedge \{ \theta_1, \dots, \theta_n \}
\end{equation}
of a rank $n$ polynomial ring (with generators $x_1, \dots, x_n$) with a rank $n$ exterior algebra
(with generators $\theta_1, \dots, \theta_n$).  
The terminology here comes from physics where the $x_i$  model the positions of bosons
and the $\theta_i$ model the positions of fermions.

The symmetric group $\symm_n$ acts on $\Omega_n$ by permuting the $x_i$ and $\theta_i$ simultaneously.
The  differentiation action $f \odot g$ of $\QQ[\xx_n]$ on itself extends to an action
$\QQ[\xx_n] \otimes \Omega_n \rightarrow \Omega_n$ of $\QQ[\xx_n]$
on superspace by acting on the first tensor factor.

Let $\varepsilon_n := \sum_{w \in \symm_n} \sign(w) \cdot w$ be the antisymmetrizing element of $\QQ[\symm_n]$.
For any $k \leq n$ and any  length $k$ sequence 
$\aaa = (a_1, \dots, a_r)$ of nonnegative integers, the {\em $\aaa$-superspace Vandermonde}
was defined in \cite{RW} to be the following element of $\Omega_n$:
\begin{equation}
\delta_n(\aaa) := \varepsilon_n \cdot ( x_1^{a_1} x_2^{a_2} \cdots x_r^{a_r} x_{r+1}^{n-r-1} \cdots x_{n-1}^1 x_n^0
\times \theta_1 \theta_2 \cdots \theta_{r} ).
\end{equation}
This reduces to the usual Vandermonde when $\aaa = \varnothing$ is the empty sequence.
Rhoades and Wilson \cite{RW} defined $V_n(\aaa)$ to be the smallest subspace of $\Omega_n$
containing $\delta_n(\aaa)$ which is closed under the differentiation action of $\QQ[\xx_n]$.
The vector space $V_n(\aaa)$ is a graded $\symm_n$-module.

\begin{theorem}
\label{superspace-result}
(Rhoades-Wilson \cite{RW})
Let $k \leq n$ be positive integers and let $\lambda = (1^k, 0^{s-k})$ be the partition with $k$ copies of $1$
and $s-k$ copies of $0$, for some $s \geq k$.
Let $\aaa = ((s-1)^{n-k})$ be the constant sequence with $n-k$ copies of $s-1$.

The graded $\symm_n$-module $V_n(\aaa)$ is isomorphic to $R_{n,\lambda}$ after grading reversal
and twisting by the sign representation.
\end{theorem}

\begin{question}
\label{superspace-question}  
By considering more general parameters $\aaa$,
can the superspace Vandermondes $\delta_n(\aaa)$ be used to give models for the  
quotient rings $R_{n,\lambda}$ for more general partitions $\lambda$
 as subspaces of $\Omega_n$?
\end{question}

In the situation of Theorem~\ref{superspace-result} when $k = s$,
Rhoades and Wilson defined \cite{RW} an extension
$\mathbb{V}_n(\aaa)$ of the module $V_n(\aaa)$ by introducing $n$ new commuting variables $y_1, \dots, y_n$
and closing under polarization operators.
The space $\mathbb{V}_n(\aaa)$ is a doubly graded $\symm_n$-module, and it was conjectured \cite{RW}
that its bigraded character is given by the symmetric function
$\Delta'_{e_{k-1}} e_n$ appearing in the Delta Conjecture \cite{HRW} of Haglund, Remmel, and Wilson.
A solution to Question~\ref{superspace-question} could lead to representation-theoretic models
for more general delta operators $\Delta'_{s_{\lambda}}$ corresponding to Schur functions $s_{\lambda}$
(see \cite{HRS2} for more details on these operators).

\section{Acknowledgements}
\label{Acknowledgements}

The authors are grateful to Sean Griffin and Andy Wilson for many helpful conversations.
This project was performed as a Research Experience for Undergraduates at UC San Diego
in 2019-2020.
B. Rhoades was partially supported by NSF Grant DMS-1500838.


\begin{thebibliography}{99}
 
 
 \bibitem{Artin} E. Artin. {\em Galois Theory}, Second edition. Notre Dame Math Lectures,
 no. 2. Notre Dame: University of Notre Dame, 1944.


 \bibitem{BG}  N. Bergeron and A. Garsia.
On certain spaces of harmonic polynomials. In:
{\em Hypergeometric functions on domains of positivity, Jack polynomials,
and applications (Tampa, FL, 1991),} 51--86,
Contemp. Math.
{\bf 138}, {\em Amer. Math. Soc., Providence, RI}, 1992.





\bibitem{GP} A. M. Garsia and C. Procesi.
On certain graded $S_n$-modules and the $q$-Kostka polynomials.
{\it Adv. Math.}, {\bf 94 (1)} (1992), 82--138.

\bibitem{Griffin} S. Griffin.
{\em Ordered set partitions, Garsia-Procesi modules, and rank varieties.}
Ph.D. Dissertation, University of Washington, 2020.


\bibitem{HHL}  J. Haglund, M. Haiman, and N. Loehr.  
A combinatorial formula for the Macdonald polynomials. 
{\it J. Amer. Math. Soc.}, {\bf 18} (2005),
735--761.



\bibitem{HRW}  J. Haglund, J. Remmel, and A. T. Wilson.  The Delta Conjecture.  
{\it Trans. Amer. Math. Soc.}, {\bf 370} (2018), 4029--4057.

\bibitem{HRS}  J. Haglund, B. Rhoades, and M. Shimozono.  Ordered
set partitions, generalized coinvariant algebras, and the Delta Conjecture.
{\it Adv. Math.}, {\bf 329} (2018), 851--915.


\bibitem{HRS2}  J. Haglund, B. Rhoades, and M. Shimozono.  
Hall-Littlewood expansions of Schur delta operators at $t = 0$.
{\it S\'em. Loth. Comb.}, {\bf B79c}, (2019). (20 pp.)






\bibitem{PR}  B. Pawlowski and B. Rhoades.
A flag variety for the Delta Conjecture.
{\em Trans. Amer. Math. Soc.}, {\bf 372} (2019), 8195--8248.




\bibitem{Rhoades}  B. Rhoades.  Ordered set partition statistics and the Delta Conjecture.
{\it J. Combin. Theory Ser. A}, {\bf 154} (2018), 172--217.


\bibitem{RW2}  B. Rhoades and A. T. Wilson.
Line configurations and $r$-Stirling partitions.
{\em J. Comb.}, {\bf 10 (3)} (2019), 411--431.


\bibitem{RW}  B. Rhoades and A. T. Wilson.
Vandermondes in superspace.
{\em Trans. Amer. Math. Soc.}, {\bf 373} (2020), 4483--4516.


\bibitem{Tanisaki} T. Tanisaki. Defining ideals of the closures of conjugacy classes and representations of the Weyl
groups. {\it Tohoku Math. J.} {\bf 33 (4)} (1982), 575--585.










  
\end{thebibliography}
\end{document}